\documentclass[11pt]{amsart}
\usepackage[mathscr]{eucal}
\usepackage{times}
\usepackage{amsmath,amssymb,latexsym,amscd}   
\usepackage{hyperref}
\hypersetup{colorlinks=true,linkcolor=blue,citecolor=brown,linktocpage=true}
\usepackage{graphicx}
\usepackage[all,cmtip]{xy}
\usepackage{fancyhdr}
\usepackage{mathalfa}
\usepackage{mathrsfs}
\usepackage{upgreek}
\usepackage{tikz}
\usepackage{graphicx}
\usepackage{tikz-cd}

\DeclareMathOperator{\Mod}{-Mod}

\DeclareMathOperator{\ann}{ann}

\DeclareMathOperator{\op}{op}

\DeclareMathOperator{\Sat}{Sat}

\newcommand{\F}{\EuScript{F}}
\newcommand{\D}{\EuScript{D}}
\newcommand{\ess}{\leq^\text{ess}}

\theoremstyle{plain}
\newtheorem*{theorem*}{Theorem}
\newtheorem{thm}{Theorem}[section]
\newtheorem{cor}[thm]{Corollary}
\newtheorem{lem}[thm]{Lemma}
\newtheorem{prop}[thm]{Proposition}

\theoremstyle{definition}
\newtheorem{dfn}[thm]{Definition}
\newtheorem{obs}[thm]{Remark}
\newtheorem{ej}[thm]{Example}
\begin{document}

\title{A point-free version of torsionfree classes and the Goldie torsion theory}  
\dedicatory{Dedicated to the professors Francisco Raggi, José Ríos Montes, and \\ the Mexican School of Rings and Modules that they formed and continue to inspire.}

\author[Medina]{Mauricio Medina-B\'arcenas}

\address{Centro de Ingenier\'ia y Desarrollo Industrial (CIDESI), Av. Playa Pie de la cuesta 702, Desarrollo San Pablo 76125, Satiago de Quer\'etaro, Qro. M\'exico}
\email{inv.asoc27mmedina@cidesi.edu.mx}

\author[Sandoval]{Martha Lizbeth Shaid Sandoval-Miranda}

\address{Departamento de Matemáticas, Universidad Aut\'onoma Metropolitana, Unidad Iztapalapa. Av San Rafael Atlixco No.186 Col.Vicentina
C.P. 09340 Del. Iztapalapa, Ciudad de México, México}
\email{marlisha@xanum.uam.mx}

\author[Zaldivar]{\'Angel Zald\'ivar-Corichi}

\address{Department of Mathematics,\\
		University of Guadalajara\\
		Blvd. Gral. Marcelino García Barragán 1421, Olímipica, 44430, Guadalajara, Jalisco\\
		M\'{e}xico.}
\email{luis.zaldivar@academicos.udg.mx}

\begin{abstract}
Torsion theories are a pinnacle in the theory of abelian categories. They are a generalization of torsion abelian groups and in this generalization one of the most studied is that whose torsionfree class consists of  
nonsingular modules. To introduce the concept of singular interval we use the symmetric idea of torsion theories, that is the torsion class determines the torsionfree class and vice-versa, thus to introduce nonsingular intervals over an upper-continuous modular complete lattice, (a.k.a idiom, a.k.a modular preframe) we define the concept of \emph{division free} set. We introduce the division free set of nonsingular intervals which defines a division set of singular intervals in a canonical way. Several properties of division free sets and some consequences of nonsingular intervals are explored allowing us to develop a small part of a point-free nonsingular theory.
\end{abstract}

\maketitle
\section{Introduction}

The Mexican algebra school has a long and interesting history, in particular, we can find the branch initiated by the members of the Instituto de Matemáticas (UNAM) Francisco Raggi, José Ríos Montes and Luis Colavita in rings and module categories. Their work began at that Institute, with the famous publication series Annals of the Institute\footnote{The last issue of this Annals was published in 1993}. In several papers \cite{colavita1978filtros}, \cite{colavita1977filtros}, \cite{raggi1966loc}, \cite{raggi1976fil}, \cite{raggi1976fil2}, \cite{raggi1976fil3}, \cite{raggi1983algunas}, \cite{raggi1987proper}, \cite{rios1977fil4}, \cite{rios1983filtros}, Raggi, Ríos and Colavita presented an entire systematic theory to study Gabriel filters and at the same time, they began a deep study of hereditary torsion theories. It can be seen that they put much effort into understanding a particular hereditary torsion theory called \emph{Goldie's torsion theory}  \cite{raggi1989some}. These papers can be considered pioneers in lattice methods applied to the classification of rings. 
Their methods were mostly applied to the frame of hereditary torsion theories but their scope goes further.   
Currently, we can consider that there are several working groups studying different aspects of ring theory from various perspectives. One of these points of view is the present manuscript, which is based on and inspired by the articles published in the Annals by Raggi, Ríos, and Colavita.
We believe that this part of Mexican algebra's history gave rise to an alive school that continues to diversify with different techniques and approaches as this work shows.

One of the main paradigms in the study of rings and module categories, or more generally Grothendieck categories, is the fact that 
for every object, the collection of all its subobjects constitutes an idiom (i.e. an upper-continuous modular complete lattice). 
In recent years a systematic procedure has been explored as a classification device to explore these structures \cite{medina2021morgan}, \cite{medina2015generalization}, \cite{perez2019boolean}.

It is said that {\it a theory is an amalgam of analogies,} following this idea there exists a suggestive analogy between some point-free techniques and some module-theoretical ones \cite{johnstone1986stone},\cite{picardo2012frames}. To be specific, we are talking about the frame of nuclei over an idiom $A$ and the analogy between them (as localizations of $A$) and localizations (as a certain kind of nuclei on the idiom of left ideals $\Lambda(R)$) of the category of modules over an associative ring with unit $R.$ That is, there exists a bijective correspondence between:

\begin{itemize}
\item[(1)] the frame of nuclei $j\colon\Lambda(R)\rightarrow\Lambda(R)$ such that $j(I:r)=(j(I):r)$ for all $r\in R$ and $I\in\Lambda(R)$, and
\item[(2)] the frame of localizations of $R\Mod$.
\end{itemize}

It is well-known that the localizations of $R\Mod$ are in one-to-one correspondence with hereditary torsion theories, that is, pairs of $Hom(\_,\_)-$orthogonal classes $(\mathcal{T},\mathcal{F})$ where $\mathcal{T}$ is closed under submodules, quotients, direct sums and extensions, and $\mathcal{F}$ is closed under submodules, direct products, injective hulls and extensions. The class $\mathcal{T}$ is called a \emph{torsion class} and $\mathcal{F}$ is called a \emph{torsionfree class}. See \cite{golan1986torsion} and \cite{simmons1988semiring} for details.

In \cite{simmons1989separating}, H. Simmons introduced an analogy of hereditary torsion classes in the theory of idioms and he proved that these objects in the idiom $\Lambda(R)$ are, effectively in bijective correspondence with the nuclei on $\Lambda(R).$ However, there has not been an analogous definition corresponding to torsionfree classes.  In the study of idioms,  we have realized that we need such a concept to get a satisfactory definition of a singular interval as a generalization of the concept of singularity in ring and module theory.

Many of the results we expose in this manuscript came out as an attempt to find an adequate definition of what we have named \emph{division free set}. So, with this concept, we can introduce the theory of nonsingular intervals in Section \ref{sec5}. The division free sets helped us to show how the nonsingular intervals resemble many of the properties of nonsingular modules.

The organization of this paper is as follows: Section \ref{pre} consists of all the background material that is needed to read the whole manuscript. In section \ref{sec3} we give the general definition of a division free set and various of its properties are investigated. In particular, Theorem \ref{OK} is key to developing the theory after it. Moreover with the division free sets we describe a method to obtain the supremum of any family of nuclei over an idiom, in particular this can be applied to the nuclei on any frame.
In Section \ref{sec4}  we introduce \emph{stable nuclei} as an analogy of the stable torsion theories. Here we make a connection with the condition $(C_1)$ in lattices which comes from module theory (Proposition \ref{c11}). Later in Section \ref{sec5}, we define nonsingular intervals (using the concept of a division free set). We explore some properties of singular and nonsingular modules in the context of idioms. Moreover, a pair of aplications of our theory to nonsingular modules are given. In this section, we also define the analogous of a TTF-class (Torsion-Torsionfree class), namely DDF-set, and show when a DDF-set gives a decomposition in the idiom. 
Finally in Section \ref{sec6} we introduce the \emph{interval of quotients} and give some properties of them using the results in Section \ref{sec5}. The set of fixed points of a nucleus is thought of as a quotient category. In module theory, one classical question was to give conditions in a ring to get a semisimple calssical ring of quotients which was answered by A. Goldie. In Theorem \ref{ssid} we give conditions in an idiom in order to get that the interval of quotients defined by the Goldie's nucleus is complemented.

\section{Preliminaries}\label{pre}
This section gives the fundamental background for the rest of the manuscript.

\subsection{Some point-free techniques}\label{subpoint}

Let us start with the crucial definition:

\begin{dfn}\label{id}
A complete lattice $(A,\leq,\bigvee,\bigwedge, 0, 1)$ is called an \emph{idiom} if:

\begin{itemize}
\item[(i)] it is \emph{modular}, that is, \[a\leq b\Rightarrow (a\vee c)\wedge b=a\vee(c\wedge b)\leqno({\rm ML})\]
for all $a,b,c\in A$.
\item[(ii)] it is \emph{upper continuous}, that is,
\[a\wedge \left(\bigvee X\right)=\bigvee\{a\wedge x\mid x\in X\}\leqno({\rm IDL})\]
for all $a\in A$ and $X\subseteq A$ directed.
\end{itemize}
\end{dfn}

An important class of idioms are the \emph{frames}.

\begin{dfn}\label{frame}
A \emph{frame} (local, complete Heyting algebra), $(A, \leq, \bigvee, \wedge, 1, 0)$ is a complete lattice that satisfies
\[a\wedge \left(\bigvee X\right)=\bigvee\{a\wedge x\mid x\in X\}\leqno({\rm FDL})\]
for all $a\in A$ and $X\subseteq A$ any subset.
\end{dfn}

Frames are the algebraic version of a topological space. Indeed, if $S$ is a topological space its topology $\mathcal{O}(S)$ is a frame.
It happens that an idiom $A$ is a frame exactly when $A$ is a distributive lattice. Our primordial example of idiom comes from ring and module theory. Given a module $M\in R\Mod$, let $\Lambda(M)$ denote the set of all submodules of $M$. It is clear that $\Lambda(M)$ constitutes a complete lattice where suprema are not unions, moreover, the following distributive laws hold: 
\[N\cap \left(\sum X\right)=\sum\{N\cap K\mid K\in X\}\] for any $N\in\Lambda(M)$ and $X\subseteq\Lambda(M)$ directed, and
\[K\leq N\Rightarrow (K+ L)\cap N=K+(L\cap N)\] for all $K,L,N\in\Lambda(M)$. Therefore $\Lambda(M)$ is an idiom for every $R$-module $M$.
In particular, the lattice of left (right) ideals $\Lambda(R)$ of any associative right with unit $R$ is an idiom.

An idiom morphism is a monotone function that commutes with arbitrary suprema and finite infima. A quotient of an idiom $A$ is a surjective idiom morphism $A\rightarrow B$. As in the case for frames (locales),  the quotients of an idiom are in bijective correspondence with its nuclei.


\begin{dfn}\label{quot}
A \emph{nucleus} on an idiom $A$, is a monotone function $j\colon A\rightarrow A$, such that:
\begin{itemize}
\item[(1)]
$a\leq j(a)$ for all $a\in A$.
\item[(2)]
$j(a\wedge b)\geq j(a)\wedge j(b)$.

\item[(3)]
$j^{2}=j$.
\end{itemize}

\end{dfn}

We will denote by $N(A)$ the set of nuclei on $A$. There exist various important facts about this partially ordered set.

\begin{thm}[\cite{simmmons1989near}]
\label{0}
Let $A$ be an idiom then, the complete lattice of all nuclei $N(A)$ is a frame.
\end{thm}

\begin{obs}
\label{1}
For every $j\in N(A)$ the set $A_{j}$, of elements fixed by $j$ is an idiom, therefore many properties of $A$ are reflected in $A_{j}$ via the surjective idiom morphism $j^{*}\colon A\rightarrow A_{j}$ given by $j^{*}(a)=j(a)$ for all $a\in A$. \end{obs}

Now what we need is to introduce a module theoretic construction of $N(A)$ (here the analogy mentioned in the introduction). Recall that given $a\leq b\in A$, the \emph{interval} $[a,b]$ is the set $[a,b]=\{x\in A\mid a\leq x \leq b\}$. Denote by $\EuScript{I}(A)$ the set of all intervals of $A$. Given two intervals $I, J$, we say that $I$ is a \emph{subinterval} of $J$,    if $I\subseteq J$, that is, if $I=[a,b]$ and $J=[a',b']$ with $a'\leq a\leq b\leq b'$ in $A$. We say that $J$ and $I$ are \emph{similar}, denoted by  $J\sim I$, if there are $l,r\in A$ with associated intervals \[L=[l,l\vee r]\quad [l\wedge r,r]=R\] such that $J=L$ and $I=R$, or $J=R$ and $I=L$. This is a reflexive and symmetric relation. Moreover, if $A$ is modular, this relation is just the canonical lattice isomorphism between $L$ and $R$. 
\begin{dfn}\label{base}
\begin{itemize}
With the same notation as above:
\item[(1)] We say that a set of intervals $\mathcal{A}\subseteq {\EuScript I}(A)$ is \emph{abstract} if it is not empty and it is closed under $\sim$. That is, 
\[J\sim I\in\mathcal{A}\Rightarrow J\in\mathcal{A}.\] 
\item[(2)]An abstract set $\mathcal{B}$ is \emph{basic} if it is closed by subintervals. That is, 
\[J\subseteq I\in\mathcal{B}\Rightarrow J\in\mathcal{B}\] 
\item[(3)] A basic set of intervals $\mathcal{C}$ is a \emph{congruence} set if it is closed under abutting intervals. That is, 
\[[a,b],[b,c]\in \mathcal{C}\Rightarrow [a,c]\in\mathcal{C}\]
for elements $a,b,c\in A$. 
\item[(4)] A basic set of intervals $\mathcal{B}$ is a \emph{pre-division} set if \[\forall x\in X\left[[a,x]\in\mathcal{B}\Rightarrow [a,\bigvee X]\in\mathcal{B}\right]\] for each $a\in A$ and $X\subseteq [a,1]$. 
\item[(5)] A set of intervals $\mathcal{D}$ is a \emph{division} set if it is a congruence set and a pre-division set.
\end{itemize}
\end{dfn}

Let $\EuScript{D}(A)\subseteq\EuScript{C}(A)\subseteq\EuScript{B}(A)\subseteq\EuScript{A}(A)$ denote the set of all division, congruence, basic and abstract sets of intervals in $A$, respectively. These gadgets can be understood like certain classes of modules in $R\Mod$, that is, classes closed under isomorphism, subobjects and quotients, extensions, and coproducts. From this point of view $\EuScript{C}(A)$ and $\EuScript{D}(A)$ are the analogy of the Serre classes and the hereditary torsion (localizations) classes in module categories.
Note that $\EuScript{B}(A)$ is closed under arbitrary intersections and unions, hence it is a frame. The top of this frame is $\EuScript{I}(A)$ and the bottom is the set of all trivial intervals of $A$, denoted by $\EuScript{O}(A)$ or simply by $\EuScript{O}$.

\begin{thm}\label{00}
If  $A$ is an idiom, then there is an isomorphism of frames 
\[N(A)\longleftrightarrow \EuScript{D}(A)\] 
given by 
\[j\longmapsto \mathcal{D}_{j}\text{ where } [a,b]\in\mathcal{D}_{j}\Longleftrightarrow b\leq j(a)\]
and 
\[\mathcal{D}\longmapsto j=\mid\mathcal{D}\mid \text{ where } \mid\mathcal{D}\mid(a)=\bigvee\left\{x\mid [a,x]\in\mathcal{D}\right\}.\]
\end{thm}

There exists a special type of nuclei. Given any family of intervals $\mathcal{C}$ since $\mathcal{D}(A)$ is a frame, there exists the least nucleus $\xi_{\mathcal{C}}$ that collapses every interval of $\mathcal{C}$. If $\mathcal{C}=\{[a,b]\}$ we will write $\xi_{[a,b]}=\xi(a,b)$.

\begin{dfn}\label{I}

\begin{itemize}
\item[(1)]An interval $[a,b]$ is \emph{simple} if there is no $a< x< b$, that is, $[a,b]=\{a,b\}$. Let $\EuScript{S}mp$ denote the set of all simple intervals.  

\item[(2)] An interval$[a,b]$ of $A$ is \emph{complemented} if it is a complemented lattice, that is, for each $a\leq x\leq b$ there exists $a\leq y\leq b$ such that $a=x\wedge y$ and  $b=x\vee y$. Let $\EuScript{C}mp$ denote the set of all complemented intervals. 

\item[(3)] An interval is $[a,b]$ \emph{weakly atomic}  if each (non-trivial) sub-interval $a\leq c <
d\leq b$ includes a simple interval, that is there is some $c\leq x < y\leq d$ with $[x,y]\in \EuScript{S}mp$  .
Let $\EuScript{WA}$ denote the set of weakly atomic intervals, an idiom is \emph{weakly atomic} if every interval is weakly atomic.
\end{itemize}
\end{dfn}


\begin{dfn}\cite[Definition 7.1]{Simmons_2010} 
Let $[a,b]$ be an interval of the idiom $A.$ An element $a\leq x\leq b$ is said to be \emph{essential in} $[a,b]$ if whenever $x\wedge y=a$ implies $y=a,$ for each $a\leq y\leq b.$ \end{dfn}


\subsection{Module theoretic preliminaries}

$\;$\\
In this section, we recall some well-known facts about the singular torsion theory or also called Goldie's torsion theory. An extensive study of singular and nonsingular modules can be found in \cite{goodearl1976ring} and \cite{goodearl1972singular}. 

Recall that, a left ideal of a ring $R$ is said to be \emph{essential} if it has nonzero intersection with every nonzero left ideal of $R$. 
Now, let $M$ be a left $R$-module. An element $m\in M$ is said to be \emph{singular} if $(0:m):=\{r\in R\mid rm=0\}$ is an essential left ideal of $R.$ Let $Z(M)$ denote the collection of all singular elements of $M$. It is not difficult to see that $Z(M)$ is a submodule of $M$ and it is called \emph{the singular submodule of $M$.} 

For the ring $R,$ $Z({}_RR)$ is a two sided ideal, called the left singular ideal of $R,$ usually denoted by $Z_{\ell}(R).$ Similarly, it can be defined the right singular ideal of $R,$ $Z_{r}(R)$. In general these two ideals are not equal. 

\begin{dfn}
A module $M$ is said to be \emph{singular} if $Z(M)$ is essential in $M$ and nonsingular if $Z(M)=0.$ 
\end{dfn}

Let $M$ be an $R$-module with $0\neq Z(M)$. Then $(0:m)$ is essential in $R$ and $Rm\cong R/(0:m)$. Thus there exists an exact sequence
\[0\to N\to F\to Z(M)\to 0\]
where $F$ is a free module and $N$ is essential in $F$. If $N$ is an essential submodule of an $R$-module $M$, then $M/N$ is singular. In general, the converse is not true, but it is true provided that the ring $R$ is nonsingular. 

The class $\mathcal{N}$ of all nonsingular modules is a torsionfree class for a hereditary torsion theory on $R\Mod$. This torsion theory is called the \emph{Goldie's torsion theory}, and it will be denoted by $\tau_g.$ The torsion radical associated to $\tau_g$ is given by \[\tau_g(M)=\left\{m\in M\mid m+Z(M)\in Z\left(\displaystyle\frac{M}{Z(M)}\right)\right\}.\]



\section{Division free sets}\label{sec3}
In this section, we introduce the concept of \emph{division free set}, we will observe that each nucleus defines a division free set and vice-versa.






\begin{prop}\cite[Lemma 5.1]{Simmons_2010}\label{chi}
Let $A$ be an idiom. For each interval $[a,b]\in\EuScript{I}(A)$ there exist the largest nucleus $\chi(a,b)$ on $A$ such that $\chi(a,b)(a)\wedge b=a$. 
\end{prop}

%
%

\begin{thm}\label{some}
	Let $A$ be an idiom.
	\begin{itemize}
		\item[(1)]
		
		$\chi(l\wedge r,r)=\chi(l,l\wedge r)$ for all $l,r\in A$.
		
		\item[(2)]
		
		$\chi(a,c)\leq\chi(a,b)$ for all $a\leq b\leq c$ in $A$.
		
		\item[(3)]
		
		$\chi(a,b)\wedge\chi(b,c)\leq\chi(a,c)$ for all $a\leq b\leq c$ in $A$.
		
		\item[(4)]
		
		$\chi(a,\bigvee X)=\bigwedge\{\chi(a,x)\mid x\in X\}$ for all $a\in A$ and all $X\subseteq A$ directed.
		
		\item[(5)]
		
		$\chi(a,\bigvee X)=\bigwedge\{\chi(a,x)\mid x\in X\}$ for all $a\in A$ and all $X\subseteq A$ independent over $a$.

		\item[(6)]
		
		$\chi(a,x)=\chi(a,b)$ for every essential element $x$ in $[a,b]\in\EuScript{I}(A)$.
	%
	\end{itemize}
\end{thm}
\begin{proof}

  The conditions (1)-(5) are stated in \cite[Theorem 6.3]{Simmons_2010}.

For the condition (6), 
	Let $a<x\leq b$ an essential element. 
	Form the definition we have $\chi(a,b)\leq\chi(a,x)$, also $\chi(a,x)(a)\wedge b\wedge x=a\wedge b=a$ then by the essential property of $x$ we have $\chi(a,x)(a)\wedge b=a$, that is, $\chi(a,x)\leq\chi(a,b)$.
\end{proof}

\begin{dfn}
	Let $A$ be an idiom. An element $a\in A$ is \emph{$\wedge$-irreducible} if whenever $x\wedge y\leq a$ implies that $x\leq a$ or $y\leq a$.
\end{dfn}

In the case $a\in A$ is $\wedge$-irreducible, we can describe the nucleus $\chi(a,1)$ in the following way.

\begin{prop}\label{chiirr}
	Let $a\in A$. If $a$ is $\wedge$-irreducible then 
	\[\chi(a,1)(x)=\begin{cases}
	a \text{ if } x\leq a \\
	1 \text{ in other case.}
	\end{cases}\]
\end{prop}

\begin{lem}
	Let $a\in A$ with $a=b\wedge c$. Then $\chi(b,1)\wedge \chi(c,1)\leq \chi(a,1)$.
\end{lem}

\begin{proof}
	It follows that
	\begin{equation*}
	\begin{split}
	\chi(b,1)\wedge\chi(c,1)(a) & =\chi(b,1)\wedge\chi(c,1)(b\wedge c) \\
	& =(\chi(b,1)\wedge\chi(c,1))(b)\wedge (\chi(b,1)\wedge\chi(c,1))(c) \\
	& =b\wedge \chi(b,1)(c) \wedge \chi(c,1)(b)\wedge c\\
	& =b\wedge c \\
	& =a.
	\end{split}
	\end{equation*}
\end{proof}

\begin{prop}
	Let $a\in A$. If $a=b\wedge c$ with $b$ and $c$ $\wedge$-irreducible elements such that $b\nleq c$ nor $c\nleq b$, then $\chi(a,1)=\chi(b,1)\wedge\chi(c,1)$.
\end{prop}

\begin{proof}
	We have that
	\[a=\chi(a,1)(a)=\chi(a,1)(b\wedge c)=\chi(a,1)(b)\wedge\chi(a,1)(c)\leq b.\]
	Since $b$ is $\wedge$-irreducible, $\chi(a,1)(b)=b$ or $\chi(a,1)(c)\leq b$. Since $c\nleq b$, $\chi(a,1)(b)=b$. Analogously, $\chi(a,1)(c)=c$. Thus $\chi(a,1)=\chi(b,1)\wedge\chi(c,1)$ by Lemma.
\end{proof}

\begin{lem}\label{diachi}
	Consider a lattice $A$ which contains a diamond lattice 
	
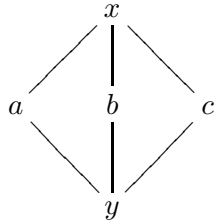
\[\xymatrix{ & x\ar@{-}[d]\ar@{-}[dl]\ar@{-}[dr] & \\ a\ar@{-}[dr] & b\ar@{-}[d] & c\ar@{-}[dl] \\ & y &}\]
	satisfying that if $a\leq z$ or $b\leq z$ or $c\leq z$ then $x\leq z$. Then $\chi(a,1)=\chi(b,1)=\chi(c,1)$.
\end{lem}

\begin{proof}
	There are few cases: Case 1: $x\leq\chi(a,1)(b),\chi(a,1)(c)$. Then 
	\[x\leq\chi(a,1)(b)\wedge\chi(a,1)(c)=\chi(a,1)(b\wedge c)=\chi(a,1)(y).\] 
	This implies that $x\leq\chi(a,1)(a)$ which is a contradiction. 
	
	Case 2: $\chi(a,1)(b)=b$ and $x\leq\chi(a,1)(c)$. Then 
	\[a=\chi(a,1)(a)\wedge\chi(a,1)(c)=\chi(a,1)(a\wedge c)=\chi(a,1)(y)=\chi(a,1)(b\wedge c)=b.\]
	Contradiction.
	
	Case 3: $\chi(a,1)(b)=b$ and $\chi(a,1)(c)=c$. This implies that $\chi(a,1)(y)=y$. Hence $\chi(a,1)$ fixes $a$, $b$ and $c$. Thus $\chi(a,1)\leq\chi(b,1)$ and $\chi(a,1)\leq\chi(c,1)$
	
	This can be done analogously with $\chi(b,1)$ and $\chi(c,1)$. Thus $\chi(a,1)=\chi(b,1)=\chi(c,1)$.
\end{proof}

\begin{dfn}\label{free}
	Let $A$ be an idiom. A subset of intervals $\EuScript{F}\subseteq\EuScript{I}(A)$ is called a \emph{division free set} or simply a \emph{free set} if:

	\begin{itemize}
		\item[(i)] It is closed under similarity, that is,
		
		\[[l\wedge r,l]\in\EuScript{F}\Leftrightarrow [r,r\vee l]\in\EuScript{F}\]
		with $r,l\in A$.

		\item[(ii)]
		
		$[a,b]\in\EuScript{F}\Rightarrow [a,x]\in\EuScript{F}$ for all $a\leq x\leq b$.
		
		\item[(iii)]
		
		$[a,b],[b,c]\in\EuScript{F}\Rightarrow [a,c]\in\EuScript{F}$ for all $a\leq b\leq c$.

		\item[(iv)]

		If $[a,b]\in\EuScript{F}$ and there exits $c\in A$ such that $b$ is essential in $[a,c]$, then $[a,c]\in\EuScript{F}$.

		\item[(v)]

		If $[x,b]\in\EuScript{F}$ for all $x\in X\subseteq A$, then $[\bigwedge X, b]\in\EuScript{F}$. 
		
	\end{itemize}
\end{dfn}

\begin{prop}\label{frees}
	Let $A$ be an idiom and $j$ a prenucleus on $A$. Then the set  \[\EuScript{F}_{j}=\{[a,b]\mid j\leq\chi(a,b)\}\] is a division free set.	
\end{prop}

\begin{proof}
	
Let $X\subseteq A$, $b\in A$ and suppose $[x,b]\in \EuScript{F}_j$ for all $x\in X$. It follows that $j(x)\wedge b=x$ for all $x\in X$. Therefore $\left(\bigwedge j(X)\right)\wedge b=\bigwedge X$. Thus,
\[\bigwedge X\leq j\left(\bigwedge X\right)\wedge b\leq \left(\bigwedge j(X)\right)\wedge b=\bigwedge X.\]
This implies that $j\leq\chi\left(\bigwedge X,b\right)$ and hence $[\bigwedge X,b]\in\EuScript{F}_j$. The other properties follow from Proposition \ref{some}. 
			
\end{proof}

\begin{obs}
    

    The set $\EuScript{F}_j$ was already defined in \cite[Definition 4.1]{sanchez2020natural} but it is not mentioned the relation with the nuclei $\chi(a,b)$. The subject in that paper is to deal with another kind of subsets of intervals called \emph{natural sets} which resemble the natural classes defined in module categories (see \cite{dauns2006classes}) but the author never makes a reference with torsion theories as we do here and his approach is different. 
\end{obs}

In the case of frames, we get the following:
Given a frame $A$, for each $a\in A$ we have the nucleus $u_a\colon A\to A$ given by $u_{a}(b)=a\vee b$.  Recall that an \emph{implication} on a complete lattice $A$, is an operation $(-\succ -)$ on $A$ defined as: $x\leq (a\succ b)$ if and only if $x\wedge a\leq b$.

\begin{cor}
    Let $A$ be a frame and $a\in A$. The following conditions are equivalent:
    \begin{enumerate}
        \item $[x,y]\in\F_{u_a}$
        \item $u_a\leq\chi(x,y)$
        \item $a\wedge y\leq x$
        \item $a\leq(y\succ x)$
        \item $y\leq(a\succ x)$
    \end{enumerate}
\end{cor}

\begin{lem}\label{partelibre}
    Let $j\in N(A)$ and $[a,b]\in \EuScript{I}(A)$. Then $[j(a)\wedge b, b]\in\EuScript{F}_j$.
\end{lem}

\begin{proof}
    Since $j$ is a nucleus,
    \[j(j(a)\wedge b)\wedge b=j^2(a)\wedge j(b)\wedge b=j(a)\wedge b.\]
    Thus, $j\leq\chi(j(a)\wedge b,b)$.
\end{proof}

\begin{prop}\label{rep1}
	Let $j\in N(A)$. Then 
	\[\bigwedge\{\chi(a,b)\mid [a,b]\in\EuScript{F}_{j}\}=j.\]
\end{prop}

\begin{proof}
	Let $k=\bigwedge\{\chi(a,b)\mid [a,b]\in\EuScript{F}_{j}\}$, clearly $j\leq k$. If $[x,y]$ is such that $y\leq k(x)$ and $x\leq j(x)\wedge y< y$ then $[j(x)\wedge y,y]$ is a nontrivial interval in $\EuScript{F}_{j}$ by Lemma \ref{partelibre}. This implies that $k\leq\chi(j(x)\wedge y,y)$ and hence $y\leq k(x)\leq \chi(j(x)\wedge y,y)(x)$. Thus $[x,y]$ is collapsed by $\chi(j(x)\wedge y,y)$. Since  
	\[\chi(j(x)\wedge y,y)(j(x)\wedge y)\wedge y=\chi(j(x)\wedge y,y)(j(x))\wedge\chi(j(x)\wedge y,y)( y)\wedge y,\]
	
	it happens that
	\begin{align*}
	j(x)\wedge y =\chi(j(x)\wedge y,y)(j(x)\wedge y)\wedge y 
	&= \\
	&=\chi(j(x)\wedge y,y)(j(x))\wedge\chi(j(x)\wedge y,y)( y)\wedge y \\
	&= \chi(j(x)\wedge y,y)(j(x))\wedge\chi(j(x)\wedge y,y)(x)\wedge y \\ 
	&= \chi(j(x)\wedge y,y)(x)\wedge y\\
	&=y.
	\end{align*}
	Therefore $j(x)\wedge y=y$ as required.
\end{proof}

\begin{prop}\label{cont}
	Let $A$ be an idiom and $j\in N(A)$. Then,
	\begin{enumerate}
		\item $\EuScript{F}_j\cap\EuScript{D}_j=\EuScript{O}$.
		\item $\EuScript{F}_{\chi(a,b)}\subseteq \EuScript{F}_j$ for all $[a,b]\in\EuScript{F}_j$.
		\item $[j(x),1]\in\EuScript{F}_j$ for all $x\in A$. 
	\end{enumerate}
\end{prop}

\begin{proof}
	(1) Let $[a,b]\in\EuScript{F}_j\cap\EuScript{D}_j$. This implies that $j(a)\wedge b=a$ and $b\leq j(a)$. Therefore, $a=b$.
	
	(2)	Let $[a,b]\in\EuScript{F}_j$ and $[x,y]\in\EuScript{F}_{\chi(a,b)}$. Then $\chi(a,b)\leq\chi(x,y)$. By hypothesis, $j\leq\chi(a,b)$. Hence $j\leq\chi(x,y)$, that is, $[x,y]\in\EuScript{F}_j$.
	
	(3) Let $x\in A$.Then $j(j(x))\wedge 1=j(x)\wedge 1=j(x)$. Hence $j\leq\chi(j(x),1)$, that is, $[j(x),1]\in\EuScript{F}_j$.
\end{proof}

\begin{prop}\label{a1}
	Let $\EuScript{F}$ be a set of intervals satisfying (i), (ii) y (iv) of Definition \ref{free}. Then, for every $[x,y]\in\EuScript{F}$ there exists $a\in A$ and $a\leq b$ such that $[x,y]$ is similar to $[a,b]$ and $[a,1]\in\EuScript{F}$.
\end{prop}

\begin{proof}
	Let $[x,y]\in\EuScript{F}$. There exists $z\in A$ such that $x\leq y\leq z$ and $y$ is essential in $[x,z]$. Then $[x,z]\in\EuScript{F}$ by (iv). On the other hand, there exists $a\in A$ such that $x\leq a$, $a\wedge z=x$ and $a\vee z$ is essential in $[a,1]$. Since $[a,a\vee z]\cong[a\wedge z,z]=[x,z]$, $[a,a\vee z]\in\EuScript{F}$ by (i). Hence $[a,1]\in\EuScript{F}$ by (iv).
\end{proof}

Last Proposition says that $\EuScript{F}$ is completely determined by the elements $a\in A$ such that $[a,1]\in\EuScript{F}$.

\begin{prop}\label{ja1}
	Let $j\in N(A)$. Then $j=\bigwedge\{\chi(a,1)\mid [a,1]\in\F_j\}$.
\end{prop}

\begin{proof}
	By Proposition \ref{rep1}, $j\leq\bigwedge\{\chi(a,1)\mid [a,1]\in\F_j\}$. Let $[x,y]\in\F_j$. It follows from Proposition \ref{a1} that there exists $a\in A$ and $a\leq b$ such that $[x,y]$ is similar to $[a,b]$ and $[a,1]\in\F_j$. By the properties of $\chi$, \[\chi(a,1)\leq \chi(a,b)=\chi(x,y).\] This implies that 
	\begin{equation*}
	\begin{split}
	\bigwedge\{\chi(a,1)\mid [a,1]\in\F_j\} & \leq\bigwedge\{\chi(a,b)\mid [a,b]\in\F_j\} \\ 
	& =\bigwedge\{\chi(x,y)\mid [x,y]\in\F_j\} \\
	& =j.
	\end{split}
	\end{equation*}
\end{proof}

\begin{cor}\label{fjfix}
	Let $j\in N(A)$ and $a\in A$. Then, $[a,1]\in\F_j$ if and only if $j(a)=a$.
\end{cor}

\begin{proof}
	$\Rightarrow$ Suppose $[a,1]\in \F_j$. By Proposition \ref{ja1}, $j\leq\chi(a,1)$. Therefore $j(a)\leq\chi(a,1)(a)=a$.
	
	$\Leftarrow$ We know that $[j(x),1]\in\F_j$ for all $x\in A$. Then $[a,1]=[j(a),1]\in\F_j$.
\end{proof}

\begin{thm}\label{OK}
	Let $A$ be an idiom and $\F$ be a set of intervals of $A$. Then, $\F=\F_j$ for some $j\in N(A)$ if and only if $\F$ satisfies (i),(ii),(iv) and (v) of Definition \ref{free}. In this case $j=\bigwedge\{\chi(a,1)\mid [a,1]\in\F\}$.
\end{thm}
\begin{proof}
	$\Rightarrow$ Suppose $\F=\F_j$. It is clear that $\F_j$ satisfies (i),(ii),(iv) and (v). 
	
	$\Leftarrow$ Define $j\colon A\to A$ as $j(a)=\bigwedge\{x\mid a\leq x\text{ and }[x,1]\in \F\}$. It is clear that, $j$ is an inflator. It follows from the property (v) that $[j(a),1]\in\F$ for all $a\in A$ and hence $j$ is idempotent. Moreover, $j(a)=a$ if and only if $[a,1]\in\F$. We claim that $j$ is a prenucleus, that is, $j(a\wedge b)=j(a)\wedge j(b)$ for all $a,b\in\F$. Let $a\in A$ and consider $\chi(j(a),1)$. Then,
	\[j(a)\leq\chi(j(a),1)(a)\leq\chi(j(a),1)(j(a))=j(a).\]
	Therefore, $j(a)=\chi(j(a),1)(a)$. Given $a,b\in A$, we have that $j(j(a\wedge b))\wedge j(a)=j(a\wedge b)\wedge j(a)=j(a\wedge b)$. Thus, $j\leq\chi(j(a\wedge b),j(a))$. Analogously, $j\leq\chi(j(a\wedge b),j(b))$.	It follows from Proposition \ref{some}.(3) that
	\begin{equation*}
	\begin{split}
	j(a)\wedge j(b) & \leq \chi(j(a\wedge b),j(a))(a)\wedge j(a)\wedge \chi(j(a\wedge b),j(b))(b)\wedge j(b) \\
	& =\chi(j(a\wedge b),j(a))(a)\wedge \chi(j(a),1)(a)\wedge \chi(j(a\wedge b),j(b))(b)\wedge \chi(j(b),1)(b)\\
	& \leq \chi(j(a\wedge b),1)(a)\wedge\chi(j(a\wedge b),1)(b)\\
	& =\chi(j(a\wedge b),1)(a\wedge b)\\
	& =j(a\wedge b).
	\end{split}
	\end{equation*}
	This proves the claim. Then $\F=\F_j$ by Proposition \ref{fjfix}.
\end{proof}

\begin{prop}\label{bij}
	For any idiom $A$ there exists a bijective correspondence among the following sets:
	
	\begin{itemize}
		\item[(i)] $N(A)$.
		
		\item[(ii)] $\mathcal{D}(A)$.
		
		\item[(iii)] $\mathcal{F}(A)$.
		
	\end{itemize}
	
\end{prop}

\begin{proof}
	
	Let us see the bijection between $N(A)$ and $\mathcal{F}(A)$.
	
	First note that if $j\in N(A)$ then $\EuScript{F}_{j}\in\mathcal{F}(A)$ and thus by Theorem \ref{OK} \[\EuScript{F}_{j}=\EuScript{F}_{j_{\EuScript{F}_{j}}}\]
	Since this implies that the nuclei $j$ and $j_{\EuScript{F}_{j}}$ have the same fixed set of points we can conclude that \[j=j_{\EuScript{F}_{j}}\]
	The the bijection is given by \[\EuScript{F}\mapsto j_{\EuScript{F}}\] and \[j\mapsto \EuScript{F}_{j}\]
	
	Then from the above and Theorem \ref{OK} we have the result.
\end{proof}

\begin{obs}
	Given a nucleus $j$, there are three manifestations of it, each one determines the other $(j,\EuScript{D}_{j},\EuScript{F}_{j})$. Also, $j$ has the following representations:

\begin{itemize}
	\item[(1)] $j=\bigwedge\{\chi(a,1)\mid [a,1]\in\EuScript{F}_{j}\}$.
	
	\item[(2)] $j=\varrho$ where $\varrho\colon A\rightarrow A$ given by $\varrho(a)=\bigwedge\{x\mid x\geq a \text{ and } [x,1]\in\EuScript{F}_{j} \}$
	
	\item[(3)] $j=\bigvee\{\xi(a,b)\mid [a,b]\in\EuScript{D}_{j}\}$.
\end{itemize}
\end{obs}

As an application of Proposition \ref{ja1} and  Theorem \ref{OK} we can give another way to describe the suprema of families of nuclei, (which in the case of frames this subject has been extensively studied, for example, see \cite{banaschewski1988another}, \cite{escardo2003joins} and \cite{johnstone1990two}), one of the difficulties of the description of suprema in $N(A)$ is the ordinal iteration that appears when one consider the least closure operator above any inflator, our procedure avoids the ordinals.

First, consider any prenucleus $k$ on an idiom $A$. By Proposition \ref{frees} the sets $\EuScript{F}_{k}$ and $\EuScript{F}_{k^{\infty}}$ are division free, in fact by corollary \ref{fjfix} \[\EuScript{F}_{k}=\EuScript{F}_{k^{\infty}},\] since $k$ and its idempotent closure $k^{\infty}$ have the same fixed elements. Then we have the nucleus $j=\bigwedge\{\chi(a,1)\mid [a,1]\in\EuScript{F}_{k}\}$ and $k\leq j$. In fact:

\begin{prop}\label{supr1}

With the notations as above: \[k^{\infty}=\bigwedge\{\chi(a,1)\mid [a,1]\in\EuScript{F}_{k}\}.\]
    
\end{prop}

\begin{proof}
 Let $j=\bigwedge\{\chi(a,1)\mid [a,1]\in\EuScript{F}_{k}\}$. For any $[a,1]\in\EuScript{F}_{k}$, we have that $k(a)=a$. Since $j\leq\chi(a,1)$, $j(a)=a$. Therefore \[\EuScript{F}_{k^\infty}=\EuScript{F}_{k}=\EuScript{F}_{j}.\]  By Proposition \ref{ja1}  we get the result.
\end{proof}

With this at hand, it is easy to describe the suprema of families of nuclei over any idiom $A$.

\begin{thm}\label{supr2}
Let $A$ be an idiom and $\mathcal{J}$ a family of nuclei on $A$. Then \[\bigvee\mathcal{J}=\bigwedge\{\chi(a,1)\mid [a,1]\in\EuScript{F}_{\dot{\bigvee}\mathcal{J}}\}.\]
\end{thm}

\begin{proof}
    Since the pointwise suprema of prenuclei is a prenucleus, then $\dot{\bigvee}\mathcal{J}$ is a prenucleus. Therefore 
    \[\bigvee\mathcal{J}=\left(\dot{\bigvee}\mathcal{J}\right)^\infty=\bigwedge\{\chi(a,1)\mid [a,1]\in\EuScript{F}_{\dot{\bigvee}\mathcal{J}}\}\]
    by Proposition \ref{supr1}.
\end{proof}

In \cite{wilson1994assembly} the author shows another way to compute the suprema of any family of nuclei over a frame, the method to do that is using the regular nuclei \[w_{a}(b)=(b\succ a)\succ a.\] Since in a frame $j\leq w_{a}$ if and only if $j(a)=a$, we have $\chi(a,1)\leq w_{a}$ and so $\chi(a,1)=w_{a}$. Thus with Theorem \ref{supr2} we obtain the method described in \cite{wilson1994assembly}.

\section{Stable nuclei}\label{sec4}
Recall that a torsion theory is \emph{stable} if the torsion class is closed under essential extensions, in this short section we give the notion of a stable nucleus and we give some properties of these nuclei. In the following section, we observed that certain nuclei are stable.
\begin{dfn}\label{sta}
	A division set $\EuScript{D}$ in an idiom $A$ is called \emph{stable} if for any $[a,b]\in\EuScript{D}$ such that there exists an $a\leq b\leq c$ for which $b$ is essential in $[a,c]$ then $[a,c]\in\EuScript{D}$. A nucleus $j\in NA$ is called \emph{stable} if $\D_j$ is stable. An idiom $A$ is \emph{stable} if every element of $N(A)$ is stable.
\end{dfn}

\begin{dfn}{\cite{albu2016conditions}}\label{C1}
	An idiom $A$ satisfies the \emph{the condition $C_1$} if for every $a\in A$, there exists a complement $c\in A$ such that $a$ is essential in $[0,c]$.
\end{dfn}

\begin{prop}\label{c11}
	Let $[a,b]\notin\F_j$ be an interval in $A$ such that $[a,b]$ satisfies $C_1$. If $j\in NA$ is stable, then $j(a)\wedge b$ has a complement in $[a,b]$.
\end{prop}

\begin{proof}
	Consider $j(a)\wedge b$. If $j(a)\wedge b$ is essential in $[a,b]$, then $[a,b]\in \D_j$ because $j$ is stable and $[a,j(a)\wedge b]\in\D_j$. Thus, $[j(a)\wedge b,b]\in\D_j\cap\F_j=\EuScript{O}$. Therefore $j(a)\wedge b=b$ which is a complement in $[a,b]$. Now, if $j(a)\wedge b$ is not essential in $[a,b]$, then there exists $x$ a maximal essential extension of $j(a)\wedge b$ in $[a,b]$ \cite[Lemma 4.3]{calugareanu2013lattice}. Since $j$ is stable $[a,x]\in\D_j$. Hence $[j(a)\wedge b,x]\in\D_j\cap\F_j=\EuScript{O}$. Thus $j(a)\wedge b=x$. It follows from \cite[Proposition 1.10(3)]{albu2016conditions} that $j(a)\wedge b$ is complemented in $[a,b]$.
\end{proof}

The stability condition on the idiom gives a strong situation (compare with \cite[Proposition 50.10]{golan1986torsion}):

\begin{prop}\label{coframe}
	Let $A$ be a stable idiom. Then $(NA)^{\op}$ is a frame.
\end{prop}

\begin{proof}
	Let $\mathcal{J}\subseteq N(A)$ and $j$ a nucleus. Since $\bigwedge\mathcal{J}\leq k$ for all $k\in\mathcal{J}$, \[j\vee(\bigwedge\mathcal{J})\leq j\vee k\] for every $k$.
 Thus \[j\vee(\bigwedge\mathcal{J})\leq \bigwedge\{j\vee k\mid k\in\mathcal{J}\}.\] To see the other comparison, let $l=\bigwedge\{j\vee k\mid k\in\mathcal{J}\}$ and $\varpi=j\vee(\bigwedge\mathcal{J})$. Suppose that we have an interval $[a,b]\in\EuScript{D}_{l}$ such that $[a,b]\notin\EuScript{D}_{\varpi}$. Then \[a\leq (\varpi)(a)\wedge b<b,\] and $[\varpi(a)\wedge b,b]\in\EuScript{F}_{\varpi}$. Therefore this interval is division free in $j$ and in $\bigwedge\mathcal{J}$. Consider $\varpi(a)\wedge b\leq c$ such that $b$ is essential in $[\varpi(a)\wedge b,c]$.
	By stability of $A$ we have $[\varpi(a)\wedge b,c]\in\EuScript{D}_{l}$. Let $\rho\in\mathcal{J}$. If $[\varpi(a)\wedge b,c]\in\EuScript{F}_{\rho}$ then $[\varpi(a)\wedge b,b]\in\EuScript{F}_{\rho}$ thus is free with respect to $j\vee\rho$ which is a contradiction since $[a,b]\in\EuScript{D}_{l}$.
	Therefore, \[\varpi(a)\wedge b<\rho(\varpi(a))\wedge b)\wedge c\leq c.\] Since the interval $[\rho(\varpi(a))\wedge b)\wedge c, c]$ is free with respect to $\rho$, it is free with respect to $\rho\vee j$. Hence  $[\rho(\varpi(a))\wedge b)\wedge c, c]\in\EuScript{F}_{l}$ which is not the case, thus we must have that $[\varpi(a)\wedge b,c]\in\EuScript{D}_{\rho}$. It implies that $[\varpi(a)\wedge b,c]\in\EuScript{D}_{\bigwedge\mathcal{J}}$ which is a contradiction.
\end{proof}

\section{Singularity}\label{sec5}

We introduce an important class of intervals called \emph{nonsingular} and \emph{singular}. This notion is central in module categories therefore in this investigation is not an exception.

\begin{dfn}\label{nonsin}
	Let $[a,b]\in\EuScript{I}(A)$. We will say that $[a,b]$ is \emph{nonsingular} if \[[c,d]\sim [a,x]\] with $c$ essential in $[0,d]$ and $a\leq x\leq b$ then $c=d$. Let  $\mathfrak{G}(A)$ denote the set of all nonsingular intervals of $A$.
\end{dfn}

%
%

\begin{obs}\label{nonidiom}
    In contrast with ring theory, where a ring can be singular or nonsingular, here from Definition \ref{nonsin} for every idiom the interval $A=[0,1]$ is always nonsingular. Therefore, we can think of our theory is an analogy with the theory of modules over nonsingular rings.
\end{obs}

\begin{prop}\label{non}
	
	The set $\mathfrak{G}(A)$ is a free set.
	
\end{prop}

\begin{proof}
	It is clear that $\mathfrak{G}(A)$ satisfies (i), (ii), (iii) of Definition \ref{free}.
	Let us verify the rest of the items of Definition \ref{free}.
	
	If $[a,b]\in\mathfrak{G}(A)$ and suppose $b$ is essential in $[a,c]$ for some $c\in A$, now consider any $a\leq x\leq c$ such that $[a,x]\cong [w,z]\cong [a,x]$ with $w$ essential in $[0,z]$, thus $a\leq b\wedge x\leq b$ therefore $[a,b\wedge x]\cong [w,y]$ with $w\leq y\leq z$ since $[a,b]\in\mathfrak{G}(A)$ we have $w=y$ and then $a=b\wedge x$ therefore $a=x$ in consequence $w=z$.
	
For the last item, suppose we have $[x,b]\in\mathfrak{G}(A)$ for each $x\in X$ and \[[\bigwedge X,y]\cong [c,d],\] with $c$ essential in $[0,d]$, consider the configuration $\bigwedge X\leq x\wedge y\leq y\leq b$ and $x\leq x\vee y\leq b$, since $[x\wedge y,y]\cong [x,x\vee y]$ therefore the interval $[x,x\vee y]$ is isomorphic to a subinterval of $[c,d]$ with least element essential and this for every $x\in X$. Then $[\bigwedge X,y]$ is trivial and thus $[c,d]$ is trivial as required.
\end{proof}

\begin{dfn}\label{goldies}
	Let $A$ be an idiom. The \emph{Goldie's nucleus} on $A$ is the nucleus $\zeta$ given by the Theorem \ref{OK}, such that $\EuScript{F}_{\zeta}=\mathfrak{G}(A)$.
\end{dfn}

\begin{prop}\label{nonornon}
	Let $[a,b]$ be a simple interval on an idiom $A$. Then $[a,b]\in\EuScript{D}_\zeta$ or $a$ is complemented in $[0,b]$ but not both.
\end{prop}

\begin{proof}
	Suppose $[a,b]\notin \EuScript{D}_\zeta$. This implies that $\zeta(a)\wedge b=a$, since $[a,b]$ is simple and $b\nleq \zeta(a)$. Therefore $\zeta\leq\chi(a,b)$ which implies that $[a,b]\in\F_\zeta$. Let $c$ be a pseudocomplement of $a$ in $[0,b]$. Then $a\vee c$ is essential in $[0,b]$. We have that $a\leq a\vee c\leq b$, so $a=a\vee c$ or $b=a\vee c$. If $a=a\vee c$, then $a$ is essential in $[0,b]$ which is a contradiction because $[a,b]\in\F_\zeta$. Thus $b=a\vee c$, proving that $a$ is complemented in $[0,b]$. Now, suppose that $[a,b]\in \EuScript{D}_{\zeta}$ and $a$ is complemented in $[0,b]$. Since $[a,b]\notin\F_\zeta$, there exists $[u,v]\cong[a,b]$ with $u$ essential in $[0,v]$. By hypothesis, $u$ must be complemented in $[0,v]$. Thus, $u=v$ which is a contradiction.
\end{proof}

\begin{cor}
    Let $c$ be a coatom of an idiom $A$. Then $c$ has a complement in $A$ if and only if $[c,1]\notin\EuScript{D}$.
\end{cor}

\begin{lem}\label{chi01}
	Let $A$ be an idiom and consider $\chi(0,1)$. Then, $a$ is essential in $[0,\chi(0,1)(a)]$ or $\chi(0,1)(a)=a$ for all $a\in A$. In particular $\chi(0,1)\leq\zeta$.
\end{lem}

\begin{proof}
	Suppose that $a\neq\chi(0,1)(a)$. Let $0\leq b\leq\chi(0,1)(a)$ such that $a\wedge b=0$. Then 
	\[0=\chi(0,1)(a\wedge b)=\chi(0,1)(a)\wedge\chi(0,1)(b)=\chi(0,1)(b).\]
	Thus $b=0$ and so $a$ is essential in $[0,\chi(0,1)(a)]$.
\end{proof}

\begin{thm}\label{lambeck}
	Let $A$ be an idiom. Then $\zeta=\chi(0,1)$ and 
	\[\EuScript{D}_\zeta=\{[a,b]\mid a\;\text{is essential in}\;[0,b]\}.\]
\end{thm}

\begin{proof}
	It is clear that, if $[0,1]\in\F_\zeta$, then $\zeta\leq\chi(0,1)$. By Lemma \ref{chi01} we have the equality. Set $\mathcal{S}=\{[a,b]\mid a\;\text{is essential in}\;[0,b]\}$. Let $[a,b]\in\mathcal{S}$ and consider $\zeta(a)\wedge b$. If $\zeta(a)\wedge b=a$, then $\zeta\leq\chi(a,b)$. This implies that $[a,b]\in\F_\zeta$ which is a contradiction. Hence $a<\zeta(a)\wedge b\leq b$. We always have that $[\zeta(a),1]\in\F_\zeta$. Then $[\zeta(a)\wedge b,b]\cong[\zeta(a),\zeta(a)\vee b]\in\F_\zeta$. Since $a$ is essential in $[0,b]$, so is $\zeta(a)\wedge b$. This is a contradiction, and so $\zeta(a)\wedge b=b$. Therefore $[a,b]\in\EuScript{D}_\zeta$. Now consider an interval $[a,b]\in\EuScript{D}_\zeta$. Suppose that there exists $0\leq c\leq b$ such that $a\wedge c=0$. Since $\EuScript{D}_\zeta$ is a division set and $[0,1]$ is nonsingular, $[a,a\vee c]\cong[0,c]\in\EuScript{D}_\zeta\cap\F_\zeta$. Thus $c=0$ and $a$ is essential in $[0,b]$. Therefore $[a,b]\in\mathcal{S}$.
\end{proof}

\begin{ej}
    Consider the following idiom $A$
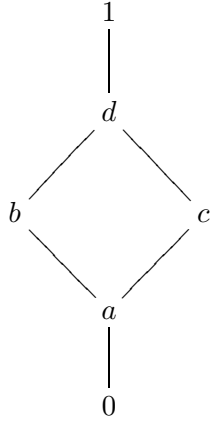
    \[\xymatrix{ & 1\ar@{-}[d] & \\ & d\ar@{-}[dl]\ar@{-}[dr] & \\ b\ar@{-}[dr] & & c\ar@{-}[dl]\\ & a\ar@{-}[d] & \\ & 0 & }\]

    It follows from Theorem \ref{lambeck}, that $[a,b],[a,c]\in\EuScript{D}_\zeta$. By similarity, $[c,d],[b,d]\in\EuScript{D}_\zeta$ and by Proposition \ref{nonornon}, $[d,1]\in\EuScript{D}_\zeta$. Threfore $\EuScript{D}_\zeta$ consists of all subinteervals of $[a,1]$. By Remark \ref{nonidiom}, $[0,a]$ is the only nonsingular interval of the idiom $A$.
\end{ej}

\begin{prop}\label{zetacom}
	Let $A$ be an idiom. Then $A_\zeta$ is complemented.
\end{prop}

\begin{proof}
	Let $a,b\in A$ with $b$ a pseudocomplement of $a$. Then $0=a\wedge b$ and $a\vee b$ is essential in $A$. It follows that $\zeta(0)=\zeta(a)\wedge\zeta(b)$ and $\zeta(a\vee b)=1$. Thus, $\zeta(a)$ is complemented in $A_\zeta$.
\end{proof}

\begin{prop}\label{zpc}
	Let $A$ be an idiom. Then $\zeta(a)=a$ if and only if $a$ is a pseudocomplement in $A$.
\end{prop}

\begin{proof}
	Let $a\in A$ such that $\zeta(a)=a$. Let $b$ be a pseudocomplement of $a$ in $A$. Let $c$ be a pseudocomplement of $\zeta(b)$ containing $a$. Hence $a$ is essential in $[0,c]$ and so $[a,c]\in\D_\zeta$. On the other hand $[a,c]=[\zeta(a),c]\in\F_\zeta$. Thus, $a=c$. Reciprocally, suppose $a$ is a pseudocomplement of $b\in A$. Then $0=\zeta(0)=\zeta(a)\wedge\zeta(b)$. Since $b\leq \zeta(b)$, $0=\zeta(a)\wedge b$. Therefore $a=\zeta(a)$.
\end{proof}

\begin{cor}\label{zetagreat}
	Let $A$ be an idiom. Then $\zeta(a)$ is the greatest element in $A$ such that $a$ is essential in $[0,\zeta(a)]$.
\end{cor}

\begin{dfn}\label{CSP}
	A complete lattice $\mathcal{L}$ satisfies \emph{the (resp. strong) complement supremum property} CSP (resp. SCSP) if the supremum of any finite (resp. any family)  of complements in $\mathcal{L}$ is a complement.
\end{dfn}

\begin{prop}\label{CSP1}
	The following conditions are equivalent for an idiom $A$:
	\begin{enumerate}
		\item $A$ satisfies $C_1$ and the CSP.
		\item $\zeta(x\vee y)=\zeta(x)\vee\zeta(y)$ for all $x,y\in A$. 
	\end{enumerate} 
\end{prop}

\begin{proof}
	$\Rightarrow$ Since $A$ satisfies $C_1$, $\zeta(a)$ has a complement in $A$ for all $a\in A$ by Proposition \ref{zpc}. Let $x,y\in A$. Then $x\vee y\leq\zeta(x)\vee\zeta(y)\leq\zeta(x\vee y)$. Note that $x\vee y$ and $\zeta(x)\vee\zeta(y)$ are essential in $[0,\zeta(x\vee y)]$. On the other hand, $\zeta(x)\vee\zeta(y)$ has a complement in $A$ by the CSP. Hence $\zeta(x)\vee\zeta(y)$ has a complement in $[0,\zeta(x\vee y)]$. This implies that $\zeta(x)\vee\zeta(y)=\zeta(x\vee y)$.
	
	$\Leftarrow$ Note that if $\zeta(x)\neq 1$, then $\zeta(x)$ cannot be essential in $[0,1]$. Let $a\in A$ such that $\zeta(a)\neq 1$.  Then, there exists a pseudocomplment $0\neq y=\zeta(y)$ of $\zeta(a)$. Hence $\zeta(a)\vee y=\zeta(a\vee y)$ is essential in $[0,1]$. This implies that $\zeta(a)\vee y=1$. Thus $\zeta(a)$ is a complement in $A$ for all $a\in A$. Hence $A$ satisfies $C_1$. Now, let $x,y\in A$ be two complements. Since $x$ and $y$ are complemets, $\zeta(x)=x$ and $\zeta(y)=y$. Therefore, $x\vee y=\zeta(x)\vee\zeta(y)=\zeta(x\vee y)$. This implies that $x\vee y$ is a complement in $A$. Thus $A$ has the CSP.
\end{proof}

\begin{dfn}\label{ddf}
	A division set $\EuScript{D}$ is called a \emph{DDF-set} (Division-Division-Free-set) if it is stable and satisfies:
		\[(\forall x\in X)\left([x,a]\in\EuScript{D}\Rightarrow [\bigwedge X,a]\in\EuScript{D}\right).\]
	A nucleus $j\in N(A)$ is \emph{DDF} if $\D_j$ is a DDF-set.
\end{dfn}

\begin{prop}\label{propjans}
	Let $A$ be an idiom and $j\in N(A)$. If $j$ is DDF, then $\neg\neg j=j$.
\end{prop}

\begin{proof}
	Since $\D_j$ is a DDF-set, there exists $k\in N(A)$ such that $\D_j=\F_k$. Since $\EuScript{O}=\D_k\cap\F_k=\D_k\cap\D_j$, we have that $k\leq\neg j$. On the other hand, if $[a,b]\in\D_{\neg\neg j}$, then $[a,\neg j(a)\wedge b]\in\D_{\neg j}\cap\D_{\neg\neg j}=\EuScript{O}$. Hence $\neg j(a)\wedge b=a$ which implies that $\neg j\leq\chi(a,b)$. Therefore $[a,b]\in\F_{\neg j}$. Since $k\leq\neg j$, $\F_{\neg j}\subseteq\F_k=\D_j$. Thus, $[a,b]\in\D_j$ and so $\D_{\neg\neg j}\subseteq\D_j$. We always have that $\D_j\subseteq\D_{\neg\neg j}$, hence $\D_j=\D_{\neg\neg j}$.
\end{proof}

If $\mathcal{T}$ is a stable hereditary torsion theory on $R\Mod$ then $\langle M\rangle(\mathcal{T})$ is a stable division set for every module $M$. The same occurs for the torsion torsion-free case.

\begin{lem}\label{goldista}
	Let $A$ be an idiom. Then the Goldie's nucleus $\zeta$ is stable. Moreover, every nucleus $j$ such that $\zeta\leq j$ is stable.
\end{lem}

\begin{proof}
	We only need to verify the second statement, to that end consider any non trivial $[a,b]\in\EuScript{D}_{j}$ and an essential extension of it, that is, $a\leq b\leq c$ such that $b$ is essential in $[a,c]$ by construction we have $[b,c]\in\EuScript{D}_{\zeta}$ then by hypothesis $[a,b],[b,c]\in\EuScript{D}_{j}$ thus $[a,c]\in\EuScript{D}_{j}$ as required.
\end{proof}

\begin{dfn}\label{goldma}
	The \emph{Goldman's nucleus} $\xi_{sp}$ is the nucleus associated to the division set generated by the nonsingular simple intervals. 
\end{dfn}

Let $\zeta$ be the Goldie's nucleus. It follows that 
\[\neg\zeta=\bigwedge\left\lbrace \chi(a,b)\mid [a,b]\in\D_\zeta\cap Smp\right\rbrace. \]
\[\neg\neg\zeta=\bigwedge\left\lbrace \chi(c,d)\mid [c,d]\in\D_{\neg\zeta}\cap Smp\right\rbrace=\bigwedge\left\lbrace \chi(c,d)\mid [c,d]\in\F_{\zeta}\cap Smp\right\rbrace. \]

\begin{prop}\label{goldi}
	Let $A$ be a weakly atomic idiom. Then, $\neg\zeta=\xi_{sp}$.
\end{prop}

\begin{proof}
	It is clear that $\xi_{sp}\leq\neg\zeta$. Let $[a,b]\in\D_{\neg\zeta}$. If $x$ is an essential element in $[a,b]$, then $[u,b]\in\D_\zeta\cap\D_{\neg\zeta}=\EuScript{O}$. Hence $[a,b]$ is complemented. By the hypothesis, $[a,b]$ is semisimple, that is, there exists an independent family $X$ of elements in $[a,b]$ such that $\bigwedge X=a$, $\bigvee X=b$ and $[a,x]$ is simple for all $x\in X$. Let $c$ be a pseudocomplement of $a$ in $[a,b]$. Then $a\vee c$ is essential in $[0,b]$. Therefore $[a\vee c,b]\in\D_\zeta\cap\D_{\neg\zeta}=\EuScript{O}$. Thus $a\vee c=b$. This implies that, for every $x\in X$, $a$ has a complement in $[0,x]$. Hence $[a,x]\in\F_{\zeta}$ for all $x\in X$, in particular $[a,x]\in\D_{\xi_{sp}}$ for all $x\in X$. Then $[a,b]=[a,\bigvee X]\in\D_{\xi_{sp}}$. Thus $\neg\zeta\leq\xi_{sp}$.
\end{proof}

\begin{prop}\label{gold}
	Let $\zeta$ be the Goldie's nucleus and suppose that $\D_\zeta$ is a DDF-set. Then $\zeta=\xi_{\bigwedge S}$ where $S=\left\lbrace s\in A\mid \zeta(s)=1\right\rbrace$. 
\end{prop}

\begin{proof}
	It follows from Proposition \ref{a1} that $\D_\zeta$ is completely determined by the intervals $[a,1]\in\D_\zeta$. Hence, $[a,1]\in\D_\zeta$ if and only if $\zeta(a)=1$. Since $\D_\zeta$ is closed under subintervals and by the hypothesis, $[\bigwedge_{s\in S}S,1]\in D_\zeta$ and every interval in $\D_\zeta$ is similar to a subinterval of $[\bigwedge_{s\in S}S,1]$. Thus, $\zeta=\xi\left( [\bigwedge_{s\in S}S,1]\right) =\xi_{\bigwedge S}$
\end{proof}

\begin{cor}\label{goldisoc}
	The following conditions are equivalent for an (resp. weakly atomic) idiom $A$:
	\begin{enumerate}
		\item $\D_\zeta$ is DDF.
		\item $\zeta=\xi_{cdb(0)}$ (resp. $\zeta=\xi_{soc(0)}$).
		\item $cbd(0)$ (resp. $soc(0)$) is essential in $[0,1]$.
	\end{enumerate}
\end{cor}

\begin{proof}
	Note that if $A$ is weakly atomic then $cbd=soc$.
	
	$\Rightarrow$ It follows from Proposition \ref{gold} that $\zeta=\xi_{\bigwedge S}$ with $S=\{s\in A\mid \zeta(s)=1\}$. If $s\in S$, then $s$ is essential in $[0,\zeta(s)]=[0,1]$. Hence $cbd(0)\leq s$. On the other hand, if $a\in A$ is essential in $[0,1]$ then $\zeta(a)=1$, that is, $a\in S$. Therefore $cbd(0)=\bigwedge S$. 
	
	$\Rightarrow$ If $\zeta=\xi_{cdb(0)}$, then $[cbd(0),1]\in\D_\zeta$. This implies that $cbd(0)$ is essential in $[0,1]$.
	
	$\Rightarrow$ Suppose that $cdb(0)=\bigwedge\{a\in A\mid a\;\text{is essential in }[0,1]\}$ is essential in $[0,1]$. Let $X\subseteq A$ such that $[x,b]\in\D_\zeta$ for all $x\in X$. Then $x$ es essential in $[0,b]$ for all $x\in X$. Let $c\in A$ be a pseudocomplement of $b$ in $A$. Then $x\vee c$ is essential in $[0,1]$ for all $x\in X$. Hence $cbd(0)\leq\bigwedge_{x\in X} (x\vee c)$. It follows that $cbd(0)\wedge b=\bigwedge_{x\in X}(a\vee c)\wedge b=\bigwedge_{x\in X}(x\vee (c\wedge b))=\bigwedge X$ is essential in $[0,b]$. Thus $[\bigwedge X,b]\in\D_\zeta$.
\end{proof}

\begin{cor}\label{goldiart}
	Let $A$ be an Artinian idiom. Then $\D_\zeta$ is a DDF-set.
\end{cor}

%
%
%

\begin{lem}\label{negproj}
	Let $A$ be an idiom. If $[a,b]\in\D_{\neg\zeta}$, then $a$ has a complement in $[0,b]$.
\end{lem}

\begin{proof}
	Suppose $[a,b]\in\D_{\neg\zeta}$. Let $c$ be a pseudocomplement of $a$ in $[0,b]$. Then $a\vee c$ is essential in $[0,b]$. This implies that $[a\vee c,b]\in\D_{\zeta}\cap\D_{\neg\zeta}=\EuScript{O}$. Thus $a\vee c=b$.
\end{proof}

\begin{lem}\label{ddfcap}
	Let $A$ be an idiom. Then $\F_{\neg\zeta}$ is a division set. Moreover, if $k\in N(A)$ is such that $\F_{\neg\zeta}=\D_k$, then $k$ is the lowest DDF nucleus above $\zeta$.
\end{lem}

\begin{proof}
	Let $[a,b]\in\F_{\neg\zeta}$ and $a\leq x\leq b$. Consider the interval $[x,\neg\zeta(x)\wedge b]\in\D_{\neg\zeta}$. By Lemma \ref{negproj} $x$ has a complement in $[0,\neg\zeta(x)\wedge b]$. Therefore $x$ has a complement in $[a,b]$. Let $c$ be a complement of $x$ in $[a,b]$. It follows that $[a,c]\sim[x,\neg\zeta(x)\wedge b]\in\F_{\neg\zeta}\cap\D_{\neg\zeta}=\EuScript{O}$. Thus $x=\neg\zeta(x)\wedge b$. This implies that $[x,b]\in\F_{\neg\zeta}$. Thus, $\F_{\neg\zeta}$ is closed under subintervals. It follows that $\F_{\neg\zeta}$ is a division set.
	
	Let $k\in N(A)$ be the nucleus such that $\F_{\neg\zeta}=\D_k$. Suppose that $\D_\zeta\subseteq\D_j$ with $\D_j=\F_l$ a DDF-set. Then,
	\[\D_l\cap\D_\zeta\subseteq\D_l\cap\D_j=\D_l\cap\F_l=\EuScript{O}.\]
	Thus $l\leq\neg\zeta$. Hence $\D_k=\F_{\neg\zeta}\subseteq\F_l=\D_j$. 
\end{proof}

\begin{prop}\label{goldgold}
	Let $A$ be a weakly atomic idiom. Then $\zeta$ is DDF if and only if $\neg\xi_{sp}=\zeta$.
\end{prop}

\begin{proof}
	$\Rightarrow$ It follows from Proposition \ref{goldi} and Proposition \ref{propjans}.
	
	$\Leftarrow$ Let $k$ be lowest DDF nucleus above $\zeta$ (given by Lemma \ref{ddfcap}) and let $[a,b]\in Smp\cap\F_{\zeta}$. If $[a,b]\notin\F_k$, then $[a,b]\in\D_k=\F_{\neg\zeta}$. By definition of the Goldman's nucleus and Proposition \ref{goldi}, $[a,b]\in\D_{\xi_{sp}}\cap\F_{\neg\zeta}=\D_{\neg\zeta}\cap\F_{\neg\zeta}=\EuScript{O}$. Contradiction. Thus $Smp\cap\F_{\zeta}\subseteq\F_k$. This implies that $k\leq\neg\xi_{sp}$. Now, if $\zeta=\neg\xi_{sp}$, then $\zeta\leq k\leq\neg\xi_{sp}=\zeta$. Thus $\zeta=k$ and so $\zeta$ is DDF.
\end{proof}

\begin{ej}\label{exa1}
	Consider the ring $R=\mathbb{Z}_2\rtimes(\mathbb{Z}_2\oplus\mathbb{Z}_2)$ the trivial extension of $\mathbb{Z}_2$ by $\mathbb{Z}_2\oplus\mathbb{Z}_2$. The ring $R$ is commutative and the lattice $A$ of ideals is given by the following diagram
	
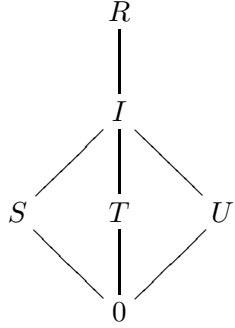
\[\xymatrix{ & R\ar@{-}[d] & \\ & I\ar@{-}[d]\ar@{-}[dl]\ar@{-}[dr] & \\ S\ar@{-}[dr] & T\ar@{-}[d] & U\ar@{-}[dl] \\ & 0 &}\]
	There are 12 nontrivial intervals on $A$. It can be seen that $\D_{\zeta}=\{[I,R]\}=\F_{\chi(I,R)}$, then $\D_\zeta$ is a DDF-set. In this case the element given by Proposition \ref{gold} is $I$ and so $\zeta=\xi_I$. On the other hand, the free set determined by $\zeta$ is
	\[\F_\zeta=\{[0,S],[0,T],[0,U],[S,I],[T,I],[U,I],[0,I],[0,R],[S,R],[T,R],[U,R]\}.\]
	Then $\zeta=\chi(S,R)\wedge\chi(T,R)\wedge\chi(U,R)\wedge\chi(0,R)=\chi(0,R)$. Now, since $\D_\zeta=\F_{\chi(I,R)}$ and $I$ is $\wedge$-irreducible, it follows from Proposition \ref{chiirr} that
	\[\D_{\chi(I,R)}=\{[0,S],[0,T],[0,U],[S,I],[T,I],[U,I],[0,I]\}.\]
	Hence, we have the tuple $(\D_{\chi(I,R)},\D_\zeta=\F_{\chi(I,R)},\F_\zeta)$.
	
	If now we consider the hereditary torsion class $\mathcal{T}$ given by the Goldie's torsion theory, then $\left\langle R\right\rangle(\mathcal{T})=\EuScript{I}(A)$ because $R$ is a singular ring. Then $\D_\zeta$ is contained properly in the slice of $\mathcal{T}$.  
\end{ej}

\begin{ej}\label{ex2}
	Let consider $\widehat{A}=A^{op}$ with $A$ as in the previous example. 
	
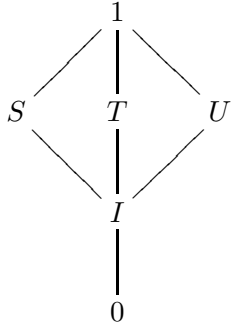
\[\xymatrix{ & 1\ar@{-}[d]\ar@{-}[dl]\ar@{-}[dr] & \\ S\ar@{-}[dr] & T\ar@{-}[d] & U\ar@{-}[dl] \\  & I\ar@{-}[d] & \\ & 0 &}\]
	Then $\widehat{A}$ is a uniform lattice. This implies that 
	\[\D_\zeta=\{[I,U],[I,T],[I,S],[I,1],[S,1],[T,1],[U,1]\}\]
	and in this case, $D_\zeta$ is a DDF too. We have that $\D_\zeta=\F_{\chi(S,1)\wedge\chi(T,1)\wedge\chi(U,1)}=\F_{\chi(S,1)}$ by Lemma \ref{diachi} and $\zeta=\xi_I$ by Proposition \ref{gold}. On the other hand, 
	\[\F_\zeta=\{[0,1],[0,I],[0,S],[0,T],[0,U]\}=\F_{\chi(0,1)}.\]
	Finally,
	\[\D_{\chi(S,1)}=\{[0,I]\}\]
	which is a division set which is not a free set. Hence we have the tuple $(\D_{\chi(S,1)},\D_\zeta=\F_{\chi(S,1)},\F_\zeta)$.
	The lattice $\widehat{A}$ corresponds to the injective hull $E$ of the simple $R$-module $R/I$ where $R$ is the ring in the previous example. Again if we consider the hereditary torsion class $\mathcal{T}$ of all singular $R$-modules, then $\left\langle E\right\rangle(\mathcal{T})=\EuScript{I}(\widehat{A})$.
\end{ej}

\begin{prop}\label{exago}
	Let $R$ be a ring and let $\mathcal{T}_g$ be the hereditary torsion class of all singular $R$-modules. The following conditions are equivalent:
	\begin{enumerate}
		\item $R$ is nonsingular.
		\item $\left\langle R\right\rangle(\mathcal{T}_g)=\D_\zeta$ where $\zeta$ is the Goldie's nucleus on $\Lambda(R)$.
	\end{enumerate}
\end{prop}

\begin{proof}
	$\Rightarrow$ Since $R$ is nonsingular, $I/J\in\mathcal{T}_g$ if and only if $I\ess J$ for every ideals $I,J\in\Lambda(R)$. Thus $\left\langle R\right\rangle(\mathcal{T}_g)=\D_\zeta$ by Theorem \ref{lambeck}.
	
	$\Leftarrow$ Let $j_R$ be the nucleus associated to the division set $\left\langle R\right\rangle(\mathcal{T}_g)$ and suppose $\left\langle R\right\rangle(\mathcal{T}_g)=\D_\zeta$. Then $\zeta=j_R$. Hence
	\[\zeta(0)=j_R(0)=\{r\in R\mid R/\ann(r)\in\mathcal{T}_g\}=\{r\in R\mid [\ann(r),R]\in\D_\zeta\}\]
	\[=\{r\in R\mid \ann(r)\ess R\}=Z(R).\]
	Since always $\zeta(0)=0$, $Z(R)=0$, that is, $R$ is nonsingular.
\end{proof}

\begin{cor}
	The following conditions are equivalent for an $R$-module $M$ over a nonsingular ring $R$:
	\begin{enumerate}
		\item $M$ is nonsingular.
		\item $\left\langle M\right\rangle(\mathcal{T}_g)=\D_\zeta$ where $\zeta$ is the Goldie's nucleus on $\Lambda(M)$.
	\end{enumerate}
\end{cor}

\begin{proof}
	$\Rightarrow$ is the same proof as in the Proposition \ref{exago}.
	
	$\Leftarrow$ Let $j_M$ be the nucleus associated to the division set $\left\langle M\right\rangle(\mathcal{T}_g)$ and suppose $\left\langle M\right\rangle(\mathcal{T}_g)=\D_\zeta$. Then $\zeta=j_M$. Hence \[\zeta(0)=j_M(0)=\{m\in M\mid R/\ann(m)\in\mathcal{T}_g\}.\]
    It follows from Proposition \ref{exago} that \[j_M(0)=\{m\in M\mid [\ann(m),R]\in\D_\zeta\}=\{m\in M\mid \ann(m)\ess R\}=Z(M).\] Since always $\zeta(0)=0$, $Z(M)=0$, that is, $M$ is nonsingular.
\end{proof}

\begin{prop}\label{morfi}
	Let $M$ and $N$ be nonsingular modules. If $\varphi:M\to N$ is a morphism, then $\varphi^{-1}\zeta_N=\zeta_M\varphi^{-1}$ where $\zeta_N$ and $\zeta_M$ are the Goldie's nucleus on $\Lambda(N)$ and on $\Lambda(M)$ respectively.
	\[\xymatrix{\Lambda(N)\ar[d]_{\zeta_N}\ar[r]^{\varphi^{-1}} & \Lambda(M)\ar[d]^{\zeta_M} \\ \Lambda(N)\ar[r]_{\varphi^{-1}} & \Lambda(M)}.\]
\end{prop}

\begin{proof}
	Let $A\in\Lambda(N)$. Since $\zeta_N(A)$ is an essential closure of $A$, $A\ess \zeta_N(A)$. Then $\varphi^{-1}(A)\ess\varphi^{-1}(\zeta_N(A))$. On the other hand, $\varphi^{-1}(A)\ess\zeta_M(\varphi^{-1}(A))$. Therefore $\varphi^{-1}\zeta_N(A)\leq\zeta_M\varphi^{-1}(A)$. Note that the quotient $\zeta_M\varphi^{-1}(A)/\varphi^{-1}\zeta_N(A)$ is singular and the quotient $N/\zeta_N(A)$ is nonsingular. Since 
	\[\frac{\zeta_M\varphi^{-1}(A)}{\varphi^{-1}\zeta_N(A)}\leq\frac{M}{\varphi^{-1}\zeta_N(A)}\hookrightarrow\frac{N}{\zeta_N(A)}.\]
	This implies that $\varphi^{-1}\zeta_N(A)=\zeta_M\varphi^{-1}(A)$.
\end{proof}

\begin{cor}\label{interseccion}
    Let $M$ be a nonsingular $R$-module and $N\leq M$ such that $M/N$ is nonsingular. Let $\overline{\zeta}$ denote the Goldie's nucleus on $\Lambda(M/N)$. Then
    \begin{enumerate}
        \item $\zeta_N(L\cap N)=\zeta_M(L)\cap N$ for all $L\leq M$.
        \item $\zeta_N(K)=\zeta_M(K)\cap N$ for all $K\leq N$.
        \item $\overline{\zeta}(L/N)=\frac{\zeta_M(L)}{N}$ for all $L/N\leq M/N$.
    \end{enumerate}
    
\end{cor}

We will present some useful results concerning $j$\emph{-essential} elements, most of this material is inspired on \cite{gomez1985spectral}.

\begin{prop}\label{jes}
	For an element $a\leq x\leq b$ the following conditions are equivalent:
	\begin{itemize}
		\item[(1)] $x$ es $j$-essential in $[a,b]$.
		\item[(2)] $j(x)\wedge b$ is $j$-essential in $[a,b]$.
		\item[(3)] $j(x)$ is essential in $[j(a),j(b)]\in\EuScript{I}(A_{j})$.
		\item[(4)] $x\vee (j(a)\wedge b)$ is essential in $[j(a)\wedge b, b]$.
		\item[(5)] $j(x)\wedge b$ is essential in $[j(a)\wedge b, b]$.   
	\end{itemize}
\end{prop}
\begin{proof}
	First (1) $\Rightarrow$ (2):
	Let $a\leq y\leq b$ such that $[a,j(x)\wedge b\wedge y]\in\EuScript{D}_{j}$ then it follows that $[a,y]\in\EuScript{D}_{j}$ since $a\leq x\leq j(x)\wedge b$.
	
	(2)$\Rightarrow$ (1).
	
	Note that if $x\wedge c\leq j(a)$ with $a\leq c\leq b$ then  $a\leq (j(x)\wedge b)\wedge c\leq j(a)\wedge b$ and since $[a,j(a)\wedge b]\in\EuScript{D}_{j}$ then the result follows. 
	
	(1) $\Leftrightarrow$ (3) is direct.
	
	(1) $\Rightarrow$ (4):
	Consider $j(a)\wedge b\leq c\leq b$ such that $(x\vee(j(a)\wedge b))\wedge c=j(a)\wedge b$ since $a\leq x\wedge c\leq j(a)\wedge b$ then $[a,x\wedge c]\in\EuScript{D}_{j}$ therefore $[a,c]\in\EuScript{D}_{j}$, that is, $c\leq j(a)$ so that $c\leq j(a)\wedge b$, as required.
	
	(4) $\Rightarrow $(5) 
	
	It follows since $x\vee (j(a)\wedge b)=b\wedge (j(a)\vee x)$ then \[b\wedge (j(a)\vee x\leq j(x)\wedge b\]
	(5) $\Rightarrow$ (1).
	Let $a\leq c\leq b$ such that $c\wedge x\leq j(a)$ then $c\wedge x\leq j(a)\wedge b$ thus $j(c)\wedge j(x)=j(a)\wedge b$ and $j(c)\wedge (j(x)\wedge b)=j(a)\wedge b$
	therefore by (5) and the above  $c\leq j(c)\wedge b=j(a)\wedge b\leq j(a)$.
	
\end{proof}

\begin{dfn}\label{jpse}
	
	Let $[a,b]\in\EuScript{I}(A)$ and element $a\leq c\leq b$ we will say $a\leq d\leq b$ is a $j$\emph{-pseudocomplement} or a \emph{pseudocomplement relative } to $j$ for $c$ if \[[a,c\wedge d]\in\EuScript{D}_{j}\] and if there exists $c'\in [a,b]$ such that $[a,c\wedge c']\in\EuScript{D}_{j}$ then $c'\leq c$. 
	
\end{dfn}

\begin{obs}\label{zj}
	Let $[a,b]$ an interval and $j\in N(A)$. Take any $a\leq x\leq b$ and consider \[\EuScript{X}_{x}=\{a\leq c\leq b\mid [a,x\wedge c]\in\EuScript{D}_{j}\}\] note that $a\in\EuScript{X}_{x}$ thus $\EuScript{X}_{x}$ is not empty and if $\mathcal{D}\subseteq\EuScript{X}_{x}$ is directed then by the idiom distributive law we have $\bigvee\mathcal{D}\in\EuScript{X}_{x}$ thus by Zorn's lemma there always exists a $j$-pseudocomplement of $x$ in $[a,b]$.

	Also if $a\leq b\leq c$ and $a\leq d\leq c$ then, $d$ is a $j$-pseudocomplement of $b$ if and only if $d$ is a pseudocomplement of $j(b)$ in $[j(a),j(b)]\cap A_{j}$. In particular since $d\leq j(d)$ then $j(a)=j(b)\wedge d\leq j(b)\wedge j(d)\leq j(a)$ thus $j(d)\leq d$ by maximality.
\end{obs}

\begin{prop}\label{jpc}
	Let $a\leq b\leq c$. If $a\leq d\leq c$ is a $j$-pseudocomplement of $b$ in $[a,c]$ then $d\vee b$ is $j$-essential in $[a,c]$.
\end{prop}

\begin{proof}
	By the above $d\vee j(b)$ is essential in $[j(a),j(b)]$, we will see that $j(d\vee j(b))$ is essential in $[j(a),j(b)]$.
\end{proof}

Recall that an interval $[a,b]$ is $j$\emph{-cocritical} if for all $a\leq x\leq b$ \[a=x\text{ or } [x,b]\in\EuScript{D}_{j}\]

\begin{prop}\label{crit}
	
	Let be $j$ a nucleus and $[a,b]\in\EuScript{C}rt(\EuScript{D}_{j})$. Then $[a,b]\in\EuScript{D}_{\zeta_{j}}$ or $a$ is $j$-complemented in $[0,b]$.
\end{prop}

\begin{proof}
If $[a,b]\notin\EuScript{D}_{\zeta_{j}}$ then $a=\zeta_{j}(a)\wedge b$ or $[\zeta_{j}(a)\wedge b, b]\in\EuScript{D}_{j}$, in the first case, that is, $\zeta_{j}\leq\chi(a,b)$ then we will reproduce the same reasoning as in Proposition \ref{nonornon}, if this is not the case, then $[\zeta_{j}(a)\wedge b,b]\in\EuScript{F}_{\zeta_{j}}$, then we can argue as again in \ref{nonornon}.
If $[a,b]\in\EuScript{D}_{\zeta_{j}}$ then then same argument but in the relative case apply.
\end{proof}

\section{The interval of quotients}\label{sec6}
Ring of quotients is an important element in the study of rings and module categories, in this section, we introduce an analogy of it in the context of intervals, and various remarkable aspects of interval of quotients are explored. 
\begin{lem}\label{essinj}
	Let $a\in A$, $j\in N(A)$ and suppose that $[0,a]\in\F_j$. Then, $a$ is essential in $[0,j(a)]$ if and only if $[0,j(a)]\in\F_j$.
\end{lem}

\begin{proof}
	$\Rightarrow$ Since $[0,a]\in\F_j$ and $\F_j$ is closed under essential extensions, $[0,j(a)]\in\F_j$.
	
	$\Leftarrow$ Let $0\leq b\leq j(a)$ such that $a\wedge b=0$. Then $[0,b]\sim[a,a\vee b]$. By hypothesis $[0,b]\in\F_j$. On the other hand, $[a,a\vee b]\in \D_j$ because it is a subinterval of $[a,j(a)]$. Therefore, $b=0$. Thus $a$ is essential in $[0,j(a)]$.
\end{proof}

Let $[a,b]$ be an interval in $A$ and let $j\in N(A)$. Consider the following diagram.

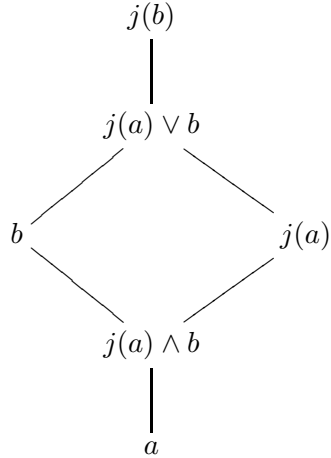
\[\xymatrix{ & j(b)\ar@{-}[d] & \\ & j(a)\vee b\ar@{-}[dr]\ar@{-}[dl] & \\ b\ar@{-}[dr] & & j(a)\ar@{-}[dl] \\ & j(a)\wedge b\ar@{-}[d] & \\ & a & }\]

We have that $[j(a)\wedge b,b]\sim [j(a),j(a)\vee b]$ and $[j(a),j(b)]$ are in $\F_j$. It follows from Lemma \ref{essinj} that $j(a)\vee b$ is essential in $[j(a),j(b)]$. Note that, $[a,b]\in\D_j$ if and only if $j(a)=j(b)$, that is, $[j(a),j(b)]$ is trivial.

\begin{dfn}
	Let $[a,b]$ be an interval in $A$ and $j\in N(A)$. The \emph{interval of quotients of $[a,b]$} is the interval $[j(a),j(b)]$ in the idiom $A_j$. This interval will be denoted by $\mathcal{Q}_j([a,b])$.
\end{dfn}

\begin{obs}
	Given nucleus $j\in N(A)$, we have that $\mathcal{Q}_j(A)=A_j$ and $\mathcal{Q}_j([a,b])$ is trivial if and only if $[a,b]\in\D_j$. Also $x$ is $j$-essential in $[a,b]$ if and only if $j(x)$ is essential in $\mathcal{Q}_j([a,b])$ by Proposition \ref{jes}.(3). 
\end{obs}

\begin{prop}
	Let $[r\wedge l,r]\sim[l,r\vee l]$ be similar intervals in $A$ and $j\in N(A)$. Then $\mathcal{Q}_j([r\wedge l,r])\sim\mathcal{Q}_j([l,r\vee l])$.
\end{prop}

\begin{proof}
	We have that $r\vee l\leq j(r)\vee j(l)\leq j(r\vee l)$. Hence $j(j(r)\vee j(l))=j(r\vee l)$. Then
	\[[j(r\wedge l),j(r)]=[j(r)\wedge j(l),j(r)]\sim[j(l),j(j(r)\vee j(l))]=[j(l),j(r\vee l)]\]
	in $A_j$. Thus $\mathcal{Q}_j([r\wedge l,r])\sim\mathcal{Q}_j([l,r\vee l])$.
\end{proof}

It follows that there is an assignation $\mathcal{Q}_j:\EuScript{I}(A)\to \EuScript{I}(A_j)$ such that $\mathcal{Q}_j(\D_j)=\EuScript{O}(A_j)$. Also, there is an assignation in the opposite direction $\mathcal{U}_j:\EuScript{I}(A_j)\to \EuScript{I}(A)$ which considers each interval in $A_j$ as an interval in $A$. It is clear that $\mathcal{Q}_j\mathcal{U}_j=Id$. On the other hand, $\mathcal{U}_j\mathcal{Q}_j([a,b])=[j(a),j(b)]\in \EuScript{I}(A)$ and this interval contains essentially an interval similar to $[j(a)\wedge b,b]$. Moreover $\mathcal{U}_j\mathcal{Q}_j(\EuScript{I}(A))\subseteq\F_j$.

\begin{lem}\label{BAj} Let $A$ be a complete Boolean algebra and 
    $j\in N(A)$. Then, 
    \[\mathcal{U}_j\mathcal{Q}_j(\EuScript{I}(A))=\F_j.\]
\end{lem}

\begin{proof}
    Let $[a,b]\in\F_j$. Then $a=j(a)$. On the other hand, $j(a)=a\vee j(0)$ because $A$ is a Boolean algebra. Hence $b=a\vee b=(j(0)\vee a)\vee b=j(0)\vee b=j(b)$. Thus $[a,b]\in\mathcal{U}_j\mathcal{Q}_j(\EuScript{I}(A))$.
\end{proof}

\begin{prop}
    Let $A$ be a frame. Then, $A$ is a complete Boolean algebra if and only if  $\mathcal{U}_{\neg\neg}\mathcal{Q}_{\neg\neg}(\EuScript{I}(A))=\F_{\neg \neg}$.
\end{prop}

\begin{proof}
    Suppose $\mathcal{U}_{\neg\neg}\mathcal{Q}_{\neg\neg}(\EuScript{I}(A))=\F_{\neg \neg}$. Let $a\in A$. Consider the interval $[0,a]$. Since $\neg\neg 0=0$, $[0,a]\in\F_{\neg\neg}$. By hypothesis, $[0,a]=[\neg\neg 0,\neg\neg a]=[0,\neg\neg a]$. This implies that $a=\neg\neg a$. Thus $A$ is a Boolean algebra. The converse follows from Lemma \ref{BAj}. 
\end{proof}

\begin{dfn}
	Let $[a,b]$ be an interval in $A$ and $j\in N(A)$. An element $a\leq x\leq b$ is called $j$-saturated if $[x,b]\in\F_j$. The set of all $j$-saturated elements of $[a,b]$ is denoted by $\Sat_j([a,b])$.
\end{dfn}

\begin{prop}
	Let $[a,b]$ be an interval in $A$ and $j\in N(A)$. Then an element $a\leq x\leq b$ is in $\Sat_j([a,b])$ if and only if $j(x)=x$. Moreover there is an isomorphism of lattices between $\Sat_j([a,b])$ and $\mathcal{Q}_j([a,b])$.
\end{prop}

\begin{proof}
	By the properties of $\F_j$, $\Sat_j([a,b])$ is closed under arbitrary infima. Thus $\Sat_j([a,b])$ is a complete lattice. Let $x\in\Sat_j([a,b])$. Then $j(a)\leq j(x)\leq j(b)$. On the other hand, $j(x)\wedge b=x$ because $[x,b]\in\F_j$. Thus, $j(x)=x$. Reciprocally, if $j(x)=x$ with $a\leq x\leq b$ then $j(x)\wedge b=x\wedge b=x$. Hence $[x,b]\in\F_j$, that is, $x\in\Sat_j([a,b])$.
\end{proof}

\begin{lem}
	Let $A$ be an idiom. Then $\mathcal{Q}_\zeta([0,u])$ is simple if and only if $[0,u]$ is uniform.
\end{lem}

\begin{proof}
	If $[0,u]$ is uniform, then $\zeta(0)=\zeta(u)$. Thus $\mathcal{Q}_\zeta([0,u])$ is simple. Conversely, let $0\leq a,b\leq u$ such that $a\wedge b=0$. This implies that $\zeta(a)\wedge\zeta(b)=\zeta(0)=0$. Suppose that $a,b\neq 0$. Since $\mathcal{Q}_\zeta([0,u])$ is simple, $\zeta(a)=\zeta(u)=\zeta(b)$. Hence $0=\zeta(a)\wedge\zeta(b)=\zeta(u)$ which is a contradiction. Thus $a=0$ or $b=0$, that is, $[0,u]$ is uniform.
\end{proof}

\begin{prop}
	Let $A$ be an idiom and suppose that every nontrivial interval contains a $\zeta$-cocritical. Then $\mathcal{Q}_\zeta([a,b])$ is semisimple for all $[a,b]$.
\end{prop}

\begin{proof}
	Let $[a,b]$ be any interval in $A$. Consider the interval $[\zeta(a),\zeta(b)]$ in $A$. By hypothesis, there exists $\zeta(a)\leq c_0\leq \zeta(b)$ such that $[\zeta(a),c_0]$ is uniform. On the other hand $\zeta(c_0)$ is complemented in $A_\zeta$. Hence there exists $d_0\in A$ such that $\zeta(d_0)$ is a complement of $\zeta(c)$ in $\mathcal{Q}_\zeta([a,b])$. Thus $\mathcal{Q}_\zeta([a,b])=[\zeta(a),\zeta(c_0)\vee \zeta(d_0)]$ in $A_\zeta$. For the interval $[\zeta(a),\zeta(d_0)]$ in $A$, there exists a $\zeta$-cocritical $\zeta(a)\leq c_1\leq \zeta(d_0)$. Since $\zeta(c_1)$ is complemented in $A_\zeta$, $\mathcal{Q}_\zeta([a,b])=[\zeta(a),\zeta(c_0)\vee \zeta(c_1)\vee \zeta(d_1)]$ in $A_\zeta$ for some $d_1\in A$. Continuing with this process, we have simple intervals $[\zeta(a),\zeta(c_i)]$ such that $\mathcal{Q}_\zeta([a,b])=[\zeta(a),\bigvee\zeta(c_i)]$.
\end{proof}

\begin{thm}\label{ssid}
	The following conditions are equivalent for an idiom $A$:
	\begin{enumerate}
		\item $A$ has finite uniform dimension,
		\item $\mathcal{Q}_\zeta(A)$ is semisimple finite.
	\end{enumerate}
\end{thm}

\begin{proof}
	By hypothesis there exists an independent finite family $\{u_1,...,u_n\}$ of elements in $A$ such that $[0,u_i]$ is uniform and $\bigvee_{i\in I} u_i$ is essential in $[0,1]$. Therefore $\zeta(\bigvee_{i\in I} u_1)=1$. This implies that $\zeta\left( \bigvee_{i\in I}\zeta(u_i)\right) =1$, that is, $1$ is the supremum of the family $\{\zeta(u_1),...,\zeta(u_n)\}$ in $A_\zeta$. Thus $\mathcal{Q}_\zeta(A)=[0,\bigvee_{i\in I}\zeta(u_i)]$ in $A_\zeta$.
	
	Reciprocally, suppose there is an independent family $\{\zeta(u_1),...,\zeta(u_n)\}$ in $A_\zeta$ such that $[0,\bigvee \zeta(u_i)]=\mathcal{Q}_\zeta(A)$ (in $A_\zeta$) with $[0,\zeta(u_i)]\in\EuScript{I}(A_\zeta)$ simple for all $1\leq i\leq n$. Then $[0,\zeta(u_i)]$ is a uniform interval in $A$. Moreover $\bigvee \zeta(u_i)$ is essential in $[0,1]$ because $[\bigvee \zeta(u_i),\zeta(\bigvee\zeta(u_i))]=[\bigvee \zeta(u_i),1]$ is in $\D_\zeta$.
\end{proof}


\bibliographystyle{acm}

\bibliography{research2}

\end{document}